\documentclass[12pt]{amsart}

\usepackage{fullpage}
\usepackage{hyperref}
\usepackage[all]{xy}

\newtheorem{theorem}{Theorem}[section]
\newtheorem{lemma}[theorem]{Lemma}
\newtheorem{proposition}[theorem]{Proposition}
\newtheorem{corollary}[theorem]{Corollary}
\newtheorem*{theorem*}{Theorem}

\theoremstyle{definition}

\newcommand{\rb}[1]{{\left( #1 \right)}}
\let\al\alpha
\let\be\beta
\let\ga\gamma
\let\si\sigma

\let\Ga\Gamma
\let\De\Delta

\def\CC{{\mathcal C}}

\def\P{{\mathbf{P}}}
\def\NP{{\mathbf{NP}}}

\def\MN{{\mathbb{N}}}
\def\MZ{{\mathbb{Z}}}

\let\too\longrightarrow
\let\bd\partial
\let\ti\tilde
\let\seq\subseteq
\let\eset\emptyset

\DeclareMathOperator\supp{supp}
\DeclareMathOperator\diam{diam}

\newcommand\Ints{{\mathbb Z}}
\newcommand\Reals{{\mathbb R}}
\newcommand\set[1]{\{ #1 \}}
\newcommand\setof[2]{\{ #1 \mid #2\}}
\newcommand\sgp[1]{\langle #1 \rangle}

\catcode`@ = 11
\def\theenumi{(\roman{enumi})}  
\def\p@enumi{\thelemma}

\let\@savedlabel\label
\def\label#1{\@savedlabel{#1}\ifnum\@listdepth=1%
\protected@edef\@currentlabel{\theenumi}\@savedlabel{#1-it@m}\fi}

\title{Spherical quadratic equations in free metabelian groups}

\author[]{Igor Lysenok}
\address{Steklov Institute of Mathematics, Gubkina str.\ 8, 119991
Moscow, Russia} \email{igor.lysenok@gmail.com}
\thanks{The first author has been partially supported by the Russian Foundation for Basic Research}

\author[]{Alexander Ushakov}
\address{Department of Mathematics, Stevens Institute of Technology,
Hoboken, NJ, 07030 USA} \email{sasha.ushakov@gmail.com}
\thanks{The second author has been partially supported by NSF grant DMS--0914773}

\begin{document}

\begin{abstract}
We prove that the Diophantine problem for spherical quadratic equations 
in free metabelian groups is solvable and, moreover, NP-complete.


\end{abstract}

\maketitle

\section{Introduction}

Let $G$ be a group, $X$ a set of variables, and $F_X$ the free group on $X$.
An {\em equation} in ~$G$ is a formal equality $W=1$ where $W \in G \ast F_X$.
A {\em solution} of an equation $W=1$ is a homomorphism $\al: G*F_X \to G$ such that $\al(W)=1$
and $\al(g)=g$ for all $g \in G$. The {\em Diophantine problem} for a class ${\mathcal C}$ of equations
 is to decide if a given equation $W=1$ in $\CC$ has a solution or not.

A word $W \in G*F_X$ and an equation $W=1$ in $G$ are {\em quadratic} if every variable $x\in X$ occurring in $W$
occurs exactly twice (as $x$ or $x^{-1}$).
Let $\CC_m$ be the class of quadratic equations of the form
$$
  z_1 c_1z_1^{-1} \ldots z_m c_mz_m^{-1} = 1, \quad c_1, \dots, c_m \in G
$$
and
$$
  \CC = \bigcup_{m=1}^\infty \CC_m
$$
We call equations in $\CC$ {\em spherical quadratic equations}. This reflects the fact that
the surface corresponding to the quadratic word $W$ is a sphere with $m$ disks removed
and thus $W$ has genus $0$ (see for example \cite{Lysenok_Myasnikov:2011}). Note that given a group $G$,
the Diophantine problem for the classes of equations ~$\CC_1$ and $\CC_2$ are the word problem and the
conjugacy problem in $G$, respectively.
Solvability of these problems in free metabelian groups was proven
in \cite{Magnus:1939} and ~\cite{Matthews:1966}. Nearly linear time algorithm
(in the length of $w$) for the word problem was found in \cite{Miasnikov_Romankov_Ushakov_Vershik:2010}
and polynomial time algorithm for the conjugacy problem was described in \cite{Vassilieva:2011}.

In this note, we consider spherical quadratic equations in the free metabelian group $M_n$ of rank $n$.
We prove that:
\begin{itemize}
    \item
for a fixed $m$,
there exists a polynomial time algorithm solving the Diophantine problem for equations in $\CC_m$.
    \item
the Diophantine problem for $\CC$ in $\NP$-complete.
\end{itemize}

\subsection*{Preliminaries: the free metabelian group}

The free metabelian group $M_n$ of rank ~$n$ is
the quotient $F_n / F_n^{(2)}$ of the
free group $F_n$ of rank $n$ over its second derived subgroup $F_n^{(2)} = [F_n', F_n']$ (where
$F_n'$ denotes the commutator subgroup $[F_n, F_n]$).
We fix a generating set $$\set{a_1, a_2, \dots, a_n}$$ for $M_n$ and by
$x \mapsto \bar x$ we denote the natural epimorphism of ~$M_n$ to the free abelian group $A_n = F_n / F_n'$.
In particular, $\set{\bar a_1, \bar a_2, \dots, \bar a_n}$ is a basis for $A_n$.

By $\Ga_n$ we denote the Cayley graph of $A_n$ with respect to the generating set $\set{\bar a_1, \dots, \bar a_n}$.
We view vertices of $\Ga_n$ as elements of $A_n$.
The graph $\Ga_n$ may be naturally viewed as the 1-skeleton $U_n^{(1)}$ of the $n$-dimensional
cubic complex $U_n$
obtained by partitioning the euclidean space ~$\Reals^n$ into the union of unit $n$-cubes.
By $C_1(\Ga_n)$ we denote the free abelian group of $1$-chains of $\Ga_n$ over $\Ints$ (or flows over $\Ints$ with finite
support in the terminology of \cite{Miasnikov_Romankov_Ushakov_Vershik:2010}),
that is, the group of all formal linear combinations
$k_1 e_1 + \dots + k_r e_r$ where $k_i\in \Ints$ and $e_i$ are edges of ~$\Ga_n$.

All $r$-chains, $r$-cycles and homology groups are taken with integer coefficients and we use for them standard
notations $C_r(X)$, $Z_r(X)$ and $H_r(X)$.

A word $w \in F_n$ determines a unique edge path $p_w$ in $\Ga_n$ labeled by $w$
which starts at ~$1$ (the vertex corresponding to the identity of $A_n$). It defines a $1$-chain
$\si(w)$ which is the algebraic sum of all edges traversed by $p_w$.
It is not hard to see that the mapping $w \mapsto \si(w)$ induces a well-defined injective map $M_n \to C_1(\Ga_n)$;
that is, two words $u$ and $w$ define the same element of $M_n$ if and only $\si(u) = \si(w)$
(see \cite{DLS,Vershik_Dobrynin:2004,Miasnikov_Romankov_Ushakov_Vershik:2010}). For $g \in M_n$,
we use the same notation $\si(g)$ for the image of $g$ under the induced map.
It is an easy exercise to check that for any $g,h \in M_n$,
\begin{equation} \label{eq:sigma-of-product}
  \si(gh) = \si(g) + \bar g \si(h)
\end{equation}
where $\bar g \si(h)$ is obtained by shifting $\si(h)$ by $\bar g$ (via the action of $A_n$ on $\Ga_n$).

For the boundary $\bd_1 \si(g)$ we obviously have
$$
    \bd_1 \si(g) = \bar g - 1.
$$
This implies that $g \in M_n'$ if and only if $\si(g)$ is a 1-cycle.
Since $\si(gh) = \si(g) + \si(h)$ if $g \in M_n'$,
we get an isomorphism between $M_n'$ and the group $Z_1(\Ga_n)$
of $1$-cycles of $\Ga_n$ (the last group coincides also with $H_1(\Ga_n)$).

\section{Spherical quadratic equations}

Let $W \in M_n * F_X$ be a word of the form
\begin{equation}\label{eq:spherical}
W = z_1 c_1 z_1^{-1} \cdot z_2 c_2 z_2^{-1} \cdot \ldots \cdot z_{m} c_{m} z_{m}^{-1}, \quad c_i \in M_n.
\end{equation}
Denote by $Q$ the set of the involved constants $c_i$,
    $$Q = \set{c_1, c_2, \dots, c_m},$$
and let
$$
  L_Q = \sgp{\bar c_1, \dots, \bar c_m}
$$
be the subgroup of $A_n$ generated by the images of $c_i$.

Let $\De_Q = \Ga_n / L_Q$ be the quotient graph over the action of $L_Q$.

The action of the group $A_n$ on $\Ga_n$ induces its action on $U_n$ and we have $(U_n/L_Q)^{(1)}=\De_Q$.
Since $U_n$ is contractible, we have the exact chain complex:
$$
  0 \to C_n(U_n) \overset{\bd_n}\too C_{n-1}(U_n)
  \overset{\bd_{n-1}}\too \dots \to C_1(U_n)
  \overset{\bd_1}\too C_0(U_n) \to 0.
$$
For $g \in M_n$, let $\tau_Q(g)$ denote the natural image of $\si(g)$ in $C_1(\De_Q)$.
Denote also by $\ga_Q$ the natural projection $U_n \to U_n/L_Q$, so $\ga_Q(\si(g))= \tau_Q(g)$ for any $g \in M_n$.

\begin{lemma}\label{le:tau_w}
If $\tau_Q(g) = 0$ then $g \in M_n'$.
\end{lemma}

\begin{proof}
Let $g$ be any element of $M_n$ such that $\tau_Q(g) = 0$.
We have $\tau_Q(g) = \ga_Q(\si(g))$ where by definition, $\si(g)$ is the 1-chain in $\Ga_n$ defined
by a path $p_{\hat g}$ starting at 1 and labeled by a word $\hat g \in F_n$ representing $g$.
Since $\bd_1$ commutes with $\ga_Q$ we have
$0 = \bd_1 \tau_Q(g) = \ga_Q(\bar g) - 1$,
that is, $\bar g \in L_Q$. This implies that $\ga_Q(p_{\hat g})$ is a loop in $\De_Q$ and
hence $0 = \tau_Q(g)$ is the image of $\ga_Q(p_{\hat g})$ under the epimorphism $\pi_1(U_n/L_Q) \to H_1(U_n/L_Q)$.
Now observe that $U_n$ is the universal cover of $U_n/L_Q$,
hence $\pi_1(U_n/L_Q) \simeq L_Q$ is abelian and
the epimorphism $\pi_1(U_n/L_Q) \to H_1(U_n/L_Q)$ is in fact an isomorphism.
Then $\ga_Q(p_{\hat g})$ represents the trivial element of the fundamental group $\pi_1(U_n/L_Q)$
and its lift $\hat g$ is a loop and hence $\si(g)$ is a cycle.
\end{proof}

For a word $W$ of the form (\ref{eq:spherical}) consider the following two series of elements in $M_n$:
\begin{itemize}
    \item[(a)]
$[[a_i, a_j]^h, c_k]$ for all $h \in M_n$, $1\le i,j\le n$ and $1\le k\le m$;
    \item[(b)]
$[c_i, c_j]$ for all $1\le i,j\le m$.
\end{itemize}
Let $H_Q$ be the subgroup generated by elements (a) and (b). Note that $H_Q$ is an abelian.
It is easy to check that for any values $\alpha,\beta,\gamma = \pm1$ an element
$[[a_i^\alpha,a_j^\beta]^u,c_k]^\gamma$ is also an element of the form (a).
For instance,
$[[a_i,a_j]^u,c_k]^{-1} = [[a_j,a_i]^u,c_k]$.

\begin{proposition} \label{prop:trace-kernel}
$\tau_Q(g) = 0$ if and only if $g\in H_Q$.
\end{proposition}

\begin{proof}
it is easy to check that $\tau_Q(g) = 0$ for every generator $g$ of $H_Q$.
Hence, sufficiency holds.

Now, assume that $\tau_Q(g) = 0$.
It follows from Lemma \ref{le:tau_w} that $g\in M_n'$. Hence $g$ can be identified
with the corresponding 1-cycle $\si(g) \in Z_1(U_n)$. Obviously, the set
    $\set{\si(g)\mid \tau_Q(g)=0}$
forms a subgroup of $Z_1(U_n)$ which is the kernel of
the homomorphism
    $\ga_*:Z_1(U_n) \to Z_1(U_n/L_Q)$
induced by the projection $\ga:U_n\to U_n/L_Q$. Our goal is to compute $\ker(\ga_*)$. To do this, we first compute
the kernel of $C_2(U_n) \overset{\bd_2}\too Z_1(U_n) \overset{\ga_*}\too Z_1(U_n/L_Q)$ and then take its image in $Z_1(U_n)$.
Since we have a commutative diagram
$$
  \xymatrix{
    C_2(U_n) \ar[r]^{\bd_2} \ar[d]^{\ga_*} &
    Z_1(U_n) \ar[d]^{\ga_*} \\
    C_2(U_n/L_Q) \ar[r]^{\bd_2} &
    Z_1(U_n/L_Q)
  }
$$
the kernel of $C_2(U_n) \overset{\bd_2\ga_*}\too Z_1(U_n/L_Q)$ is
generated by $\ker(C_2(U_n) \to C_2(U_n/L_Q))$ and a set of preimages of generators for
$\ker(C_2(U_n/L_Q) \overset{\bd_2}\too Z_1(U_n/L_Q))$.

Clearly, $\ker(C_2(U_n) \to C_2(U_n/L_Q))$ is generated by all 2-chains of the form
$$
  s - \bar c_k s, \quad s \in U_n^{(2)}.
$$
For the kernel of $C_2(U_n/L_Q) \overset{\bd_2}\too Z_1(U_n/L_Q)$, we have two types of generators.
First, there are generators for the image of $\bd_3: C_3(U_n/L_Q) \to C_2(U_n/L_Q))$.
These are boundaries $\bd_3 E$ of all 3-cubes $E \in (U_n/L_Q)^{(3)}$.
These generators vanish when mapped to $Z_1(U_n)$.
Second, there are the generators for $H_2(U_n/L_Q)$ which can be easily computed as follows.

As a topological space, $U_n/L_Q$ is homeomorphic to the direct product of a finite number of copies
of the real line $\Reals$ and circles $S^1$. Clearly, $H_2(U_n/L_Q)$ is generated by all embedded tori
$S^1 \times S^1$ for all distinct factors $S^1$. Each factor $S^1$ represents a 1-cycle
 $\ga(\rho)$ where $\rho$ is a linear combination of 1-cycles $\tau_Q(c_i)$. This implies
that $H_2(U_n/L_Q)$ is generated by the images under $\ga$ of all 2-cycles in $U_n$
with boundaries $\si([c_i,c_j])$.

Summarizing, we get the following set of generators for  $\ker(Z_1(U_n) \overset\ga\to Z_1(U_n/L_Q))$:
boundaries of all 2-cycles $s - \bar c_k s$, $s \in U_n^{(2)}$ and all 1-cycles of the form $\si([c_i,c_j])$.
The boundaries of 2-cubes $s\in U_n^{(2)}$ are of the form $\si([a_i, a_j]^h)$, $h \in M_n$
and we have
$$
  \bd_2 (s - \bar c_k s) = \si([a_i, a_j]^h) - \si(c_k [a_i, a_j]^h c_k^{-1})
  = \si([[a_i, a_j]^h, c_k]).
$$
This gives elements (a). The 1-cycles $\si([c_i,c_j])$ give elements (b).
Thus, necessity holds.
\end{proof}

\begin{lemma}\label{le:two_comm}
Let $H_Q'$ be the subgroup of $M_n$ generated by all elements (a). Then
$[[c_i,c_j],h] \in H_Q'$ for any $1\le i,j\le m$ and $h\in M_n$.
\end{lemma}

\begin{proof}
Observe that the following identities hold in $M_n$:
\begin{gather}
\label{eq:id1} [u v,x] = [u,x][v,x],\\
\label{eq:id2} [[x,y],z] = [[x,z],y]\cdot [[z,y],x],
\end{gather}
for any $u, v \in M_n'$ and $x, y, z \in M_n$.
Since $M_n'$ is abelian, we have $[u, h] = [u, h^g]$ for any $u \in M_n'$ and $g,h \in M_n$.
This obviously implies that $H_Q'$ is a normal subgroup of $M_n$. Furthermore, by  (\ref{eq:id1})
we see that $H_Q'$ contains all elements $[u,c_i]$ where $u \in M_n'$.
Finally, it follows from (\ref{eq:id2}) that
$[[c_i,c_j],h] = [[c_i,h],c_j]\cdot [[h,c_j],c_i].$
\end{proof}

Note that $\tau_Q(u_i c_i u_i^{-1})$ is a $1$-cycle in $\De_Q$ obtained by the shift
of $\tau_Q(c_i)$  by the element ~$\bar u_i$ (under the action of $A_n$).
The next proposition states
that deciding if a given spherical equation has a solution
is equivalent to finding shifts of the $1$-cycles $\tau_Q(c_i)$  in $\De_Q$ with trivial algebraic sum.

\begin{proposition} \label{pr:main-reduction}
The equation $W=1$ has a solution if and only if
there exist elements $u_1,u_2,\dots,u_m \in M_n$ satisfying
\begin{equation}\label{eq:shift_sum}
  \sum_{i=1}^m \tau_Q(u_i c_i u_i^{-1}) = 0.
\end{equation}
\end{proposition}

\begin{proof}
For $\ti g = (g_1,\dots,g_m) \in (M_n)^m$ define:
$$
  W(\ti g) = g_1 c_1 g_1^{-1} g_2c_2 g_2^{-1}  \dots  g_mc_m g_m^{-1} .
$$
Let $\ti u$ be a solution of the equation $W=1$, i.e.\ we have $W(\ti u) = 1$. 
By \eqref{eq:sigma-of-product},
\begin{equation} \label{eq:sigma-of-W}
  0 = \si(W(\ti u))
  = \si(u_1 c_1 u_1^{-1}) + \bar h_1 \si(u_2 c_2 u_2^{-1}) + \dots + \bar h_{m-1} \si(u_m c_m u_m^{-1})
\end{equation}
where $h_i = u_1 c_1 u_1^{-1} \dots u_i c_i u_i^{-1}$. By the definition of $\De_Q$, shifting
of a chain $t \in C_1(\De_Q)$ by $h \in L_Q$ does not change $t$. Hence
$$
  \bar h_{i-1} \tau_Q (u_i c_i u_i^{-1}) = \tau_Q (u_i c_i u_i^{-1})
$$
and applying $\ga_*$ to the both sides of \eqref{eq:sigma-of-W} we get \eqref{eq:shift_sum}.
This proves the ``$\Rightarrow$'' part.

Now assume that $u_1,u_2,\dots,u_m \in M_n$ satisfy (\ref{eq:shift_sum}).
Define a set:
$$
  S = S_W = \setof{W(\ti g)}{\ti g \in (M_n)^m}.
$$
Clearly, $W=1$ has a solution if and only if $1\in S$. We claim that
    $$S H_Q \seq S,$$
i.e.\ $S$ is a union of $H_Q$-cosets.
It is sufficient to prove that $Sh \seq S$ for any element ~$h$ of the form (a) or (b).

For any $\ti g = (g_1, g_2, \dots, g_m)$ and $h \in M_n$ we have
$$
  W(g_1,\ldots,g_{k-1},g_k[a_j,a_i]^{hd},g_{k+1},\ldots,g_m)= W(\ti g) \cdot [[a_i,a_j]^h,c_k]
$$
where $d = (g_{k+1} c_{k+1} g_{k+1}^{-1} \dots g_m c_m g_m^{-1})^{-1}$.
This proves $S H_Q' \subseteq S$.
Similarly,
$$
  W(g_1,\ldots,g_{j-1},c_i g_j,g_{j+1},\ldots,g_m)= W(\ti g) \cdot [c_i,c_j] \cdot [[c_i,c_j],f]
$$
where $f = g_j (g_{j+1} c_{j+1} g_{j+1}^{-1} \dots g_m c_m g_m^{-1})^{-1}.$
By Lemma \ref{le:two_comm}, $[[c_j,c_i],f] \in H_Q'$.
Using the already proved fact that $S H_Q' \subseteq S$ we see that $S h \seq S$ for any $h$
of the form (b) and hence $S H_Q = S$ as required.

As we have seen in the proof of ``$\Rightarrow$'', equality (\ref{eq:shift_sum})
can be rewritten as
$$
  \tau_Q(u_1 c_1 u_1^{-1} \dots u_m c_m u_m^{-1}) = 0.
$$
Hence, by Proposition \ref{prop:trace-kernel}, 
$u_1 c_1 u_1^{-1} \dots u_m c_m u_m^{-1} \in H_Q \cap S$. 
Now it follows from the equality $SH_Q=S$ that $1 \in S$.
Therefore, $W = 1$ has a solution.
\end{proof}

\begin{corollary} \label{co:main-reduction-improved}
The equation $W=1$ is solvable if and only if there exists a tuple
$\ti \al = (\al_1, \al_2, \dots, \al_m) \in (A_m)^n$ such that
\begin{equation} \label{eq:lattice-equation}
  \al_1 \tau_Q(c_1) + \al_2 \tau_Q(c_2) + \dots + \al_m \tau_Q(c_m) = 0.
\end{equation}
\end{corollary}

\begin{proof}
As observed above, the condition $\sum_{i=1}^m \tau_Q(u_i c_i u_i^{-1}) = 0$ in Proposition \ref{pr:main-reduction}
can be rewritten as
$$
  \tau_Q(u_1 c_1 u_1^{-1}) + \tau_Q(u_2 c_2 u_2^{-1}) + \dots + \tau_Q(u_m c_m u_m^{-1}) = 0.
$$
It remains to notice that $\tau_Q (u_i c_i u_i^{-1}) = \bar u_i \tau_Q (c_i)$ since $\tau_Q(c_i)$ is
a 1-cycle in $\De_Q$.
\end{proof}

{\em Remark:} It is not hard to see from the proof of Proposition \ref{pr:main-reduction} that tuples $\ti \al$
represent solutions of the equation $W=1$ in the following way: given a tuple $\ti\al$ satisfying
\eqref{eq:lattice-equation}, there exists a solution ~$\ti u$ of $W=1$ such that $u_i \in \al_i L_Q$,
$i=1,2,\dots,m$.

Our next step is to give an effective bound on the size of a tuple $\ti\al$ satisfying \eqref{eq:lattice-equation}.
For $\ga \in A_n$ by $|\ga|$ we denote the word length of $\ga$ in
the generators $\set{\bar a_1, \bar a_2, \dots, \bar a_n}$ of $A_n$.
For $x \in M_n$, $|x|$ denotes the word length of $x$ in the generators $\set{a_1,a_2,\dots,a_n}$.

\begin{proposition}
If a tuple $\ti \al$ satisfies \eqref{eq:lattice-equation}, then there exists another tuple
$\ti\al' = (\al_1',\al_2',\dots,\al_m') \in (A_n)^m$ satisfying \eqref{eq:lattice-equation} such that
$|\al_i'| \le 2 \sum_j |c_j|$ for all $i=1,\ldots,m$.
\end{proposition}

\begin{proof}
Denote
$$
  \tau_i = \al_i \tau_Q(c_i), \ i=1,\dots,m.
$$
Given a 1-chain $\rho \in C_1(\De_Q)$,
$$
  \rho = \sum_{i=1}^r k_i e_i, \quad k_i \ne 0, \ e_i \in E(\De_Q), \ e_i \ne e_j \text{ for } i \ne j,
$$
define
$$
  \supp (\rho) = \set{e_1,\dots,e_r}.
$$
Denote $I = \set{1,\dots,m}$.
We a call a non-empty subset $J \seq I$ a {\em cluster} if $\supp(\tau_i) \cap \supp(\tau_j) = \eset$
for any $i\in J$ and $j \notin J$. It follows from the definition that if $J_1$ and $J_2$ are clusters
with non-empty intersection then $J_1 \cap J_2$ is again a cluster. Hence $I$ can be partitioned into a
finite disjoint union of minimal clusters,
$$
    I = J_1 \uplus J_2 \uplus\dots\uplus J_t.
$$
(Another view on minimal clusters: draw a graph with the set of vertices $I$ where two vertices $i$ and $j$
are joined with an edge iff $\supp(\tau_i) \cap \supp(\tau_j) \ne\eset$. Then minimal clusters are connected
components of this graph.)

We introduce an integer-valued distance function $d(e,f)$ on the set $E(\De_Q)$ of edges of $\De_Q$.
By definition, the distance between edges $e$
and $f$ is the distance between their midpoints in the graph $\De_Q$ (where all edges are assumed to have
the length 1). For example, $d(e,f)=1$ if and only if $e$ and $f$ are distinct and have a common  vertex.
The following statements hold:
\begin{enumerate}
\item
If $J$ is a cluster, then $\sum_{i \in J} \tau_i = 0$.
\item
If $J$ is a cluster, then $\sum_{i \in J} (\al_i +\be) \tau_Q(c_i) = 0$ for any $\be\in A_n$.
\item
If $J$ is a minimal cluster, then
$$
  \diam\left( \bigcup_{i \in J} \supp(\tau_i) \right) \le \sum_{i\in J} |c_i|.
$$
\end{enumerate}
There are two types of minimal clusters: trivial and nontrivial.
We say that $J$ is trivial if $J=\{j\}$ which happens if and only if $\supp(\tau_j) = \emptyset$.
If $|J|> 1$, then it is called nontrivial. Clearly, $\supp(\tau_j)\ne\emptyset$ for any $j$
in a nontrivial cluster $J$.

Now, consider an arbitrary minimal cluster $J$. 
If $J$ is trivial, then set $\alpha_j'$ to $(0,\ldots,0)\in A_n$.
If $J$ is nontrivial, then we choose any $j\in J$ and
the origin $\beta\in A_n$ of an arbitrary edge $e\in \supp(\tau_j)$
and set $\alpha_j'$ to $\alpha_j'-\beta$ for every $j\in J$.
It follows from (ii) that so defined $\ti\al'$ satisfies (\ref{eq:lattice-equation})
and from (iii) that $|\al_i'| \le 2 \sum_j |c_j|$.
\end{proof}

\begin{proposition}
Let $n,m\in\MN$.
Given a tuple $(c_1,\ldots,c_m) \in M_n^m$ and a tuple $(\alpha_1,\ldots,\alpha_m) \in A_n^m$
satisfying $|\al_i| \le 2 \sum_j |c_j|$ it requires polynomial time in $\sum_{j=1}^m |c_j|$ 
to check (\ref{eq:lattice-equation}).
\end{proposition}

\begin{proof}
The sum in (\ref{eq:lattice-equation}) is an algebraic sum of edges in $\Delta_Q$
traversed by paths $\al_i\tau_Q(c_i)$. We clearly can construct such a sum 
in polynomial time because we can efficiently distinguish
vertices in $\Delta_Q$ using Gauss elimination.
\end{proof}

\begin{corollary}
The Diophantine problem for the class of spherical equations is in $\NP$.
\end{corollary}

\begin{proof}
For a certificate we can take a tuple $\ti\al$ satisfying \eqref{eq:lattice-equation} with
bound $|\al_i| \le 2 \sum_j |c_j|$.
\end{proof}

\begin{corollary}
The Diophantine problem for the class $\CC_m$ of spherical equations in $M_n$ with 
a bounded number of variables is in $\P$.
\end{corollary}

\begin{proof}
The size of the set of all possible certificates $\{(\al_1,\ldots,\al_m) \mid |\al_i| \le 2 \sum_{j=1}^m |c_j|\}$
can be bounded above by a polynomial $\rb{2\sum_{j=1}^m |c_j|}^{nm}.$
\end{proof}

To prove $\NP$-completeness of the Diophantine problem for equations $W=1$
we use $\NP$-completeness of the {\em Square packing puzzle problem} (see \cite{Demaine_Demaine:2007}):
\begin{itemize}
\item
{\em Input:} A tuple of positive integers $(n_1,n_2,\ldots,n_k)$ written
in the unary notation (i.e.\ the size of the input is $\sum_i n_i$).
\item
{\em Question:} Can the set of $k$ square pieces of sizes $n_1\times n_1$, $\ldots$, $n_k\times n_k$
be exactly packed into the square box of area $n_1^2 + n_2^2+\ldots+n_k^2$?
\end{itemize}

For a $k$-tuple $\ti n= (n_1,\ldots,n_k) \in\MN^k$
satisfying $n_0=\sqrt{n_1^2+\ldots+n_k^2} \in \MN$ define a word $W_{\ti n} \in M_2 * F_{\set{z_0,z_1,\dots,z_k}}$:
    $$W_{\ti n} = z_0 [a_2^{n_0},a_1^{n_0}] z_0^{-1} \cdot \prod_{i=1}^k z_i [a_1^{n_i},a_2^{n_i}] z_i^{-1}.$$
\begin{theorem}
The Diophantine problem for equations $W_{\ti n}=1$ in $M_2$ is $\NP$-complete.
\end{theorem}

\begin{proof}
Fix a tuple $\ti n$. The set $Q$ associated with $W_{\ti n}$ is
    $$Q = \{[a_1^{n_0},a_2^{n_0}]^{-1},[a_1^{n_1},a_2^{n_1}],\ldots, [a_1^{n_k},a_2^{n_k}]\}.$$
The corresponding subgroup $L_Q$ of $A_2$ is trivial and $\Delta_Q$ is simply the Cayley graph of $\MZ^2$
relative to the standard generators $\{\bar a_1,\bar a_2\}$.
Hence, for every $c=[a_1^{n_i},a_2^{n_i}]\in Q$, $\tau(c)$ is the grid $(n_i\times n_i)$-square contour
starting at $(0,0)$.
It follows from Corollary \ref{co:main-reduction-improved} that the set of 
$k$ squares given by $\ti n$ can be packed into the
$(n_0\times n_0)$-square if and only if the equation $W_{\ti n}=1$ has a solution. We get a reduction to the
square packing puzzle problem.
\end{proof}

Observe that an equation $W=1$ in $M_2$ has a solution if and only if it has a solution when viewed as
an equation in $M_n$ for $n\ge 2$. We immediately get:

\begin{corollary}
For any $n\ge 2$, the Diophantine problem for the class of spherical quadratic equations
in $M_n$ is $\NP$-complete.
\end{corollary}

\bibliography{main_bibliography}

\end{document}